\documentclass[a4paper, 12pt]{amsart}

\usepackage{amsmath, amsthm, amsfonts, amssymb}
\usepackage[hidelinks]{hyperref}
\usepackage[english]{babel}
\usepackage{mathrsfs}
\usepackage{eucal}
\usepackage[all]{xy}
\usepackage{tikz}
\usepackage{verbatim}

\usepackage{amssymb} 
\def\acts{\curvearrowright}

\newtheorem{thm}{Theorem}[section]
\newtheorem*{thm*}{Theorem}
\newtheorem{lem}[thm]{Lemma}
\newtheorem*{prob*}{Problem}

\newtheorem{prop}[thm]{Proposition}
\newtheorem*{prop*}{Proposition}

\newtheorem{cor}[thm]{Corollary}
\newtheorem*{cor*}{Corollary}

\theoremstyle{definition}
\newtheorem{defn}[thm]{Definition}
\newtheorem*{defn*}{Definition}

\newtheorem{remark}[thm]{Remark}

\newtheorem{question}[thm]{Question}

\newtheorem*{question*}{Question}
\newtheorem*{Pquestion*}{Popa's question}

\newtheorem*{conv*}{Convention}

\newcommand{\norm}[1]{{\left\lVert #1\right\rVert}}

\def\del{\partial}

\makeatletter

\def\dotminussym#1#2{%
  \setbox0=\hbox{$\m@th#1-$}%
  \kern.5\wd0%
  \hbox to 0pt{\hss\hbox{$\m@th#1-$}\hss}%
  \raise.6\ht0\hbox to 0pt{\hss$\m@th#1.$\hss}%
  \kern.5\wd0}

\DeclareMathOperator{\supp}{supp}

\newcommand{\mc}{\mathcal}
\newcommand{\mb}{\mathbb}

\begin{document}

\title{ Proper proximality among various families of groups}
\author{Changying Ding and Srivatsav Kunnawalkam Elayavalli}

\email{changying.ding@vanderbilt.edu}
\urladdr{https://cding.me}

\email{srivatsav.kunnawalkam.elayavalli@vanderbilt.edu}
\urladdr{https://sites.google.com/view/srivatsavke}
\address{Department of Mathematics\\Vanderbilt University, 1326 Stevenson Center, Station B 407807, Nashville, TN 37240}
\maketitle

\begin{abstract}
    In this paper, the notion of proper proximality (introduced in \cite{BIP18}) is studied and classified in various families of groups. We show that if a group acts non-elementarily by isometries on a tree such that for any two edges, the intersection of their edge stabilizers 
    is finite, then $G$ is properly proximal. 
    We show that the wreath product $G\wr H$ is properly proximal if and only if $H$ is non-amenable. 
    We then completely classify proper proximality among graph products of non-trivial groups. Our results generalize  recent work of Duchesne, Tucker-Drob and Wesolek classifying inner amenability for these families of groups. Our results also recover some rigidity results associated to the group von Neumann algebras, by virtue of being properly proximal. A key  idea in the proofs of our theorems is a technique to upgrade from relative proper proximality using computations in the double dual of the small at infinity boundary.  
\end{abstract}

\section{Introduction and statements of main results}

The goal of this paper is to provide several new examples of \emph{properly proximal} groups. The authors of \cite{BIP18} who introduced this property were motivated by the program of classifying group von Neumann algebras. 
Proper proximality is a dynamical/geometric property by nature, so it is independently of interest to group theorists and geometers. One advantage of proper proximality is that it applies to a robust family of non-amenable groups, including all lattices in non compact semisimple Lie groups.
This, in particular, allowed the authors of \cite{BIP18} to demonstrate the first $W^*$-strong rigidity results for compact actions of higher rank lattices:

\begin{thm}[Theorem 1.1 in \cite{BIP18}]\label{BIP theorem}
For all properly proximal groups $G$, the group von Neumann algbra $L(G)$ has no weakly compact Cartan subalgebras in the sense of Ozawa and Popa \cite{OzPo10}. Moreover, for any free ergodic probability measure preserving (p.m.p.) action $\sigma: G\acts (X, \mu)$, the crossed product $L^\infty(X, \mu)\rtimes G$ admits a weakly compact Cartan subalgebra $A$
if and only if $\sigma$ is weakly compact and, in this case, $A$ is unitary conjugate with $L^\infty(X, \mu)$.
\end{thm}

Our first main result proves proper proximality for a family of groups acting on trees. Unless otherwise mentioned, all groups in this paper are \emph{countable}.

\begin{thm}\label{main theorem action on trees}
Let an infinite countable $G$ act on a countably infinite tree $T$ such that:
\begin{enumerate}
    \item the action is non-elementary on the Bowditch compactification $\Delta T$;
    \item for any pairs of distinct edges $e,f\in E({T})$, one has $\mathrm{Stab}(e)\cap\mathrm{Stab}(f)$ is finite.
\end{enumerate} Then $G$ is properly proximal.
\end{thm}

The most natural examples of groups with the above phenomenon arise from fundamental groups of graphs of groups, via Bass-Serre theory. Particular cases of these include amalgamated free products and HNN-extensions.

A subgroup $H$ is almost malnormal in $G$ if for any $g\in G\setminus H$ one has $|gHg^{-1}\cap H|<\infty$. 

\begin{cor}\label{main theorem amalgamated product}
If $G=G_1*_HG_2$ is a countable group such that $H$ is almost malnormal in $G$ and $[G_1:H]\geq 3$, $[G_2:H]\geq 2$, then $G$ is properly proximal.

\end{cor}

\begin{cor}\label{main theorem HNN ext}
If $G*_{H\sim K}$ is an HNN-extension of a countable group $G$ over almost malnormal subgroups $H$ and $K$, and $[G:H]
\geq 3$, then $G*_{H\sim K}$ is properly proximal.
\end{cor}

\begin{remark}
A group is said to be a \emph{convergence group} if it admits a non-elementary convergence action (see \cite{Bow99}). In \cite{Dah03} and \cite{Osi04}, many amalgamated products and HNN-extensions are shown to be convergence groups and hence properly proximal (see Example 4.6 in \cite{BIP18}) . However, the above corollaries include examples that are not non-elementary convergence groups. For instance, consider $G=(G_1 \wr  \Gamma )*_{ \Gamma}( G_2\wr \Gamma)$, where $G_1$ and $G_2$ are non-trivial amenable groups and $\Gamma $ is any infinite countable group. Suppose there exists a non-elementary convergence action $G\curvearrowright K$ on some compact space $K$, then as $\oplus_{\Gamma}G_i$ is infinite amenable for $i=1,2$, there must exist $a_i,b_i\in K$ such that $\{a_i,b_i\}$ is fixed by $\oplus_{\Gamma}G_i$ set-wise. Since the normalizer of $\oplus_{\Gamma}G_i$ is $G_i\wr \Gamma$, we have $G_i\wr \Gamma$ fixes $\{a_i,b_i\}$ as well. Now taking any infinite sequence in $\Gamma$, we see from north south dynamics that $\{a_1,b_1\}=\{a_2,b_2\}$, hence, $G$ fixes $\{a_1,a_2\}$ which contradicts the fact that the action is non-elementary.
\end{remark}

\begin{remark}
We would like to point out for the particular cases of groups in Corollaries \ref{main theorem amalgamated product} and \ref{main theorem HNN ext}, the consequence of Theorem \ref{BIP theorem} is weaker  than what is already known in the literature, due to work of Ioana in \cite{Io15}, and subsequently Vaes in \cite{Vae13}. 
\end{remark}




Our second main result is the following: 

\begin{thm}\label{wreath product}
Let $G$ be a non-trivial group and $H$ a infinite group. Then $G\wr H$ is properly proximal if and only if $H$ is non-amenable.
\end{thm}

The above result in combination with Theorem \ref{BIP theorem} shows absence of weakly compact Cartan subalgebras in Bernoulli shift crossed products of non-amenable groups. This result is in the flavor of the well known open question (Problem III in \cite{Io18}) of whether such crossed products have unique Cartan subalgebra up to unitary conjugation. 

The proofs of the above two results use the double dual characterization of proper proximality to upgrade from proper proximality relative to certain malnormal subgroups. 

Our third main  result is a complete classification for graph products of non-trivial groups.  We provide the following algorithm to decide if a graph product of groups is properly proximal:

\begin{thm}\label{main theorem introduction 3}
Let $\mc{G}_{pp}$ be the collection of graph products of non-trivial groups, defined recursively: 
\begin{enumerate}
    \item If $\Gamma$ satisfies the condition that the radius of $\Gamma$ (see Section 2.1 for the definition), is at least 2 
    and there does not exist a pair of vertices $v_1,v_2\in V(\Gamma)$ such that $(u,v_i)\in E(\Gamma)$ for all $u\in V(\Gamma)\setminus \{v_1,v_2\}$, and $i=1,2$  then $\Gamma(G)\in \mc{G}_{pp}$, for any arbitrary choice of non-trivial groups $G_v$.
    \item If $r(\Gamma)=2$ and there are vertices $v_1, v_2\in V(\Gamma)$ such that $(u,v_i)\in E(\Gamma)$ for all $i\in \{1,2\}$, $u\in V(\Gamma)\setminus \{v_1,v_2\}$, and $|G_{v_i}|\geq 3$ for some $i=1,2$, then $\Gamma_S(G)\in \mc{G}_{pp}$, where $\Gamma_S$ is the subgraph induced by the vertices $S=V(\Gamma)\setminus \{v_1,v_2\}$.
    \item If $r(\Gamma)=1$ with some $v\in V(\Gamma)$ such that $(v,u)\in E(\Gamma)$ for all $u\neq v$, and $G_v$ is properly proximal, then $\Gamma_{T}(G)\in \mc{G}_{pp} $, where $\Gamma_T$ is the induced subgraph by $T=V(\Gamma)\setminus \{v\}$.
 \end{enumerate}
 Then, a graph product $\Gamma(G)$ is properly proximal if and only if $\Gamma(G)\in \mc{G}_{pp}$.
\end{thm}

Note that since proper proximality implies non-inner amenability, the above theorem generalizes Theorem 4.14 in \cite{DTW19} where the same classification result is obtained in the context of inner amenability.

Note also that from Theorem \ref{main theorem introduction 3}, we immediately deduce:
\begin{cor}
Arbitrary finite graph products of properly proximal groups are properly proximal.
\end{cor}

Rigidity properties of graph products of groups has gathered considerable interest recently. For instance, \cite{CdSW} shows  primeness results for the group von Neumann algebras associated to many graph products of groups. Also, \cite{Ca20} shows strong solidity and/or absence of Cartan for classes of Hecke von Neumann algebras associated to some graph products. As a consequence of Theorem \ref{main theorem introduction 3}, we obtain several new examples of groups satisfying the conclusion of Theorem \ref{BIP theorem}.

Moreover, if one has a combination of weak amenability and proper proximality, we can deduce absence of Cartan subalgebras and $\mc{C}$-rigidity in the sense of \cite{PoVa14a}\footnote{By $\mc{C}$-rigidity for $G$, we mean $L^{\infty}(X,\mu) \rtimes G$ admits a unique Cartan subalgebra up to unitary conjugation, for any free ergodic pmp action $\sigma: G\acts (X,\mu)$. See \cite{PoVa14a} for more details.} (see \cite[Theorem 1.5]{BIP18}). In \cite{Re17} Reckwerdt shows that graph products of weakly amenable groups with Cowling-Haagerup constant $1$ continue to be weakly amenable. Therefore we obtain:

\begin{cor}\label{absence of cartan}
Let $\Gamma(G)\in \mc{G}_{pp}$ be as in Theorem \ref{main theorem introduction 3} and further assume that $G_v$ is weakly amenable for all $v\in V(\Gamma)$ with Cowling-Haagerup constant $1$. Then $L(\Gamma(G))$ admits no Cartan subalgebras. Moreover, $\Gamma(G)$ is $\mc{C}$-rigid.
\end{cor}

We ask the following general question in the context of our result above: 

\begin{question}
Which graph products of groups do not admit Cartan subalgebras in their group von Neumann algebras?
\end{question}

After circulating an initial version of this preprint, the second author and I. Chifan answered the above question in \cite{chifan2021cartan}.

\subsection*{Acknowledgements:} The authors are very grateful to their advisor J. Peterson his encouragement and several helpful comments. They also thank D. Osin for helpful  conversations. Thanks are due to R. Boutonnet, I. Chifan and A. Ioana for taking a look at our preprint and offering valuable advice and feedback.

\section{Preliminaries}

\subsection{Graph products of groups}
Throughout this paper we denote by $\Gamma$ a non-empty finite simple (every edge connects two distinct vertices) undirected graph, with vertex set $V(\Gamma)$ and edge set $E(\Gamma)$. For $v\in V(\Gamma)$, its link is $S_v=\{w\in V(\Gamma) \mid (v,w)\in E(\Gamma)\}$. For any two vertices $v,w\in V(\Gamma)$, let $d(v,w)\in \mb{N}\cup \{\infty\}$ be the length of the shortest path between $v$ and $w$.  For any subset $T\subset V(\Gamma)$, denote the induced subgraph on $T$ by $\Gamma_T$.

The \emph{radius} of a graph $\Gamma$ is given by $$r(\Gamma)= \inf_{v\in V(\Gamma)} \sup_{w\in V(\Gamma)} d(u,v).$$
And a graph $\Gamma$ is called \emph{irreducible} if the complement graph $\Gamma^c$ (i.e, the graph consisting of the same vertices, and $(u,v)\in E(\Gamma^c)$ if and only if $(u,v)\notin E(\Gamma)$) is connected


Given a $\Gamma$ be a graph, and $\{G_v\}_{v\in V(\Gamma)}$ a family of countable groups labeled by the vertex set of $\Gamma$, the \emph{graph product} denoted by $G= \Gamma(G)$ is the quotient group of the free product $*_{v\in V(\Gamma)} G_v$, with relations $[g,h]=1$ for all $g\in G_u$ and $h\in G_v$ with $(u,v)\in E(\Gamma)$.

\begin{defn}\label{reduced words in amalgamated free products}
Consider an amalgamated free product group $G_1*_H G_2$. Choose $T_i$ a transversal for the cosets $\{Hx: x\in G_i\}$. A normal word is a word $g=ht_1\hdots t_k$ where $h\in H$, $k\geq 0$ and $t_j\in T_{i_j}\setminus \{1\}$ for some $i_j\in \{1,2\}$ and $i_j\neq i_{j+1}$ for $1\leq j\leq k-1$.  
\end{defn}
In amalgamated free products, it is well-known that every element can also be represented by a unique normal word (see \cite[Theorem 3.7]{BH99}). As is the case for free products, elements in graph products also admit normal forms.
\begin{defn}[Definition 3.5 in \cite{Gr90}]\label{definition of reduced word}
Let $G= \Gamma\{G_v\}_{v\in V(\Gamma)}$ be a graph product of groups. A word $g_1g_2,\hdots g_n\in G$ is said to be \emph{reduced} if the following hold: 
\begin{enumerate}
    \item $g_i\in G_{v_i}$ for all $i\in \{1,2,\hdots n\} $, where $v_i\in V(\Gamma)$.
    \item $g_i\neq 1$ for all $i\in \{1,2,\hdots, n\}$. 
    \item $\forall i\leq k< j$ such that $$[g_i,g_{i+1}]= [g_i,g_{i+2}]= \cdots= [g_i,g_k] =1,$$ and $$[g_{k+1},g_j]=[g_{k+2},g_j]=\cdots= [g_{j-1},g_j]=1,$$ then $v_i\neq v_j$. 
\end{enumerate}
\end{defn}

\begin{thm}[Theorem 3.9 in \cite{Gr90}]\label{normal form theorem for graph products}
Let $G= \Gamma\{G_{v}\}_{v\in V(\Gamma)}$ be a graph products of groups. Then, each non-trivial element $g\in G$ can be uniquely (up to commuting segments) expressed as a product: $$g=g_1\hdots g_n $$ where $g_1\hdots g_n$ is a reduced word.
\end{thm}

Graph products decompose naturally as amalgamated free products. We record this below in the following well-known lemma:

\begin{lem}\label{graph products are amalgamated free products}
Let $\Gamma(G)$ be a graph product and let $v\in V(\Gamma)$. Then, $G$ is cannonically isomorphic to  $ G_1 *_{H} G_2$ where $G_1= \Gamma_{S_v\cup \{v\}}(G)$, $G_2= \Gamma_{V(\Gamma)\setminus \{v\}}(G)$ and $H= \Gamma_{S_v}(G)$.   
\end{lem}



\subsection{Proper proximality}
As stated in the introduction, the family of properly proximal groups is a robust family including the following classes of groups: 
\begin{enumerate}
    \item Non-amenable bi-exact groups.
    \item Non-elementary convergence groups.
    \item Lattices in non-compact semi-simple Lie groups.
    \item Groups admitting a proper 1-cocycle into a non-amenable representation.
    \item Groups acting properly non-elementarily by isometries on proper CAT(0) spaces.
    \item Non-elementary  mapping class groups.
    \item Groups measure equivalent to any of the above.
\end{enumerate}

Items (1) to (4) are results of \cite{BIP18}, items (5) and (6) are  results of \cite{HHL20}, and item (7) is due to \cite{IsPeRu19}.




We say a sequence $\{g_n\}_{n\in\mathbb N}\in G$ goes to infinity \emph{relative} to a countable family of subgroups $\{H_i\}_{i\in I}$, denoted by $g_n\to \infty/\{H_i\}_{i\in I}$, if for any $t_1,\ t_2\in G$ and any $i\in I$, there exists $N\in \mb{N}$ such that for all $n\geq N$, $g_n\notin t_1 H_i t_2$. 

Consider the $C^*$-subalgebra $c_0(G,\mathcal S)\subset \ell^\infty(G)$ consisting of functions $f$ such that for all $g_n\to \infty/\mathcal S$, $f(g_n)\to 0$, where $\mathcal S=\{H_i\}_{i\in I}$. It contains the ideal $c_0(G)$ consisting of functions vanish at infinity. Observe also that $c_0(G,\mathcal S)$ is globally left and right $G$ invariant. Therefore $\ell^\infty(G)/c_0(G,\mathcal S)$ is isomorphic to $C(X_{\mathcal S})$, where $X_{\mathcal S}$ is a closed left and right invariant subset of the Stone-Cech boundary $\Delta G\setminus G$ and  is called the \emph{boundary piece} associated to a collection of subgroups $\mathcal S$.  


\begin{defn}[Theorem 4.3, \cite{BIP18}]\label{relative proper proximality with states}
Given a group $G$ and a family of subgroups $\mathcal S$,  we say $G$ is properly proximal relative to $\mathcal S$ if one of the following equivalent conditions holds:
\begin{enumerate}
    \item There is an action $G\curvearrowright K$ by homeomorphisms on a compact Hausdorff space $K$  such that there is no $G$-invariant probability measure on $K$ and there exists a probability measure
    $\eta\in \mathrm{Prob}(K)$ with  $X_{\mathcal S}=\partial_\eta G:=\{\omega\in \Delta G\setminus G\mid \lim_{g\to\omega} ^{wk^*}g\cdot\eta-(gh)\cdot\eta=0 \mathrm{\ for\ any\ } h\in G\}$.
    \item There are actions by homeomorphisms $G\curvearrowright K_i$, $i=1, \cdots, k$ on compact Hausdorff spaces $K_i$  such that there is no $G$-invariant measure on any $K_i$ and there exists a probability measure $\eta_i\in \mathrm{Prob}(K_i)$ with $X_{\mathcal S}\subset \cup_{i=1}^k \partial_{\eta_i}G$.

    \item There is no left-$G$ invariant state on $\left(\ell^{\infty}(G)/c_0(G,\mathcal S)\right)^{G_r}$ (i.e, the right $G$ invariant subspace of $\ell^{\infty}(G)/c_0(G,\mathcal S)$). 
    \item There is no left-$G$ invariant state on  $\left(\left(\ell^{\infty}(G)/c_0(G,\mathcal S)\right)^{**}\right)^{G_{r}}$. 
\end{enumerate}
\end{defn}
The group $G$ is properly proximal if it is properly proximal relative to the trivial subgroup.

Note that, from the above definition/theorem, if one has a finite collection of subgroups $\{H_i\}_{i=1}^n$ of $G$, such that $G$ is properly proximal relative to $\{H_i\}$ for each $i=1,\cdots, n$ where $\cup_{i=1}^{n}{X_{H_i}}= \Delta G\setminus G$, then $G$ is properly proximal.






\subsection{Groups acting on trees}

We say that an action of $G$ on a set $X$ is non-elementary if it doesn't preserve any set of cardinality at most 2. 

Given a simplicial tree $T$, let $\Delta T$ denote the compactification introduced in \cite{BOW12} (see \cite[$\mathsection\mathsection$ 5.2]{BO08} for details). 
The compactification $\Delta T$ is defined to be $V(T)\sqcup \partial T$ as a set, where $\partial T$ is the Gromov boundary. For each $x\in \Delta T$ and each finite set $F\subset V(T)$, set $U(x,F)=\{y\in \Delta T\mid [x,y]\cap F=\emptyset\}\cup \{x\}$ where $[x,y]$ denotes the unique geodesic path between $x$ and $y$.
Then $\{U(x,F)\mid F\subset V(T)\  \mathrm{finite}\}_{x\in \Delta T}$ forms a basis for the \emph{Bowditch topology} on $\Delta T$.
It is well known that $\Delta T$ equipped with this topology is compact and Hausdorff and any isometric action $G\curvearrowright T$ extends to an action by homeomorphisms on $\Delta T$.

If $G$ is an amalgamated free product or HNN-extension, then we have an action by homeomorphisms of $G\curvearrowright \Delta T$, where $T$ is the Bass-Serre tree associated to $G$ (see \cite{Ser80}).
Also note that if $T$ is countable, then the topology on $\Delta T$ is second countable hence metrizable as the basis may be taken as $\{U(x,F)\mid F\subset V(T)\  \mathrm{finite}\}_{x\in  V(T)}$.

\section{Proofs of main theorems}

\subsection{Proofs of Theorems \ref{main theorem action on trees} and \ref{wreath product}.}
The proof of Theorem \ref{main theorem action on trees}, splits into two steps. The first step is to establish proper proximality relative to the edge stabilizers, and the second step is to upgrade this to proper proximality using the malnormality condition. The way we accomplish step 1 is by obtaining a relative north-south dynamics for the action on the Bowditch compactification of the tree. For step 2, we use the double dual proximal space  $\left(\left(\ell^{\infty}(G)/c_0(G,\mathcal S)\right)^{**}\right)^{G_{r}}$  as a filler space to create several right invariant projections that are indexed by left cosets of the edge stabilizers in $\mathcal S$. Under the malnormality assumption, we see that these are orthogonal projections. From step 1, we see that any $G$-invariant state vanishes on the complement of the sum of these projections. To conclude, we $G$-equivariantly embed $C(\Delta T)$ into  $\left(\left(\ell^{\infty}(G)/c_0(G,\mathcal S)\right)^{**}\right)^{G_{r}}$ along these projections. This contradicts the fact that $G\curvearrowright \Delta T$ is non-elementary. A modification of the above step 2 in combination with a technical lemma adapted from \cite{BO08}, gives us Theorem \ref{wreath product}.

We begin by showing the relative north-south dynamics result described above. We remark that a  result in this flavor in the setting of graphs of convergence groups was recently obtained by R. Tomar in Section 4 of \cite{tomar}.

\begin{prop}\label{action on trees}
Let $G$ be a group acting on a tree $T$ with $|V(T)|=\infty$ by isometries. Set $\mathcal S=\{\mathrm{Stab}(e)\mid e\in E(T)\}$ and suppose $g_n\to\infty \slash\mathcal S$.
Then there exists a subsequence $(g_{n_k})$ and points $a$, $b$ in $\Delta T$ such that for any $x\in \Delta T\setminus\{ b\}$, we have $lim_{k\to\infty} g_{n_k}x=a$. Furthermore, if the action $G\curvearrowright \Delta T$ is non-elementary, then $G$ is properly proximal relative to $\mathcal S$.
\end{prop}
\begin{proof}

For any such a sequence $\{g_n\}$, one may pick a vertex $o\in V(T)$ such that $|\{g_n o\}|$ and $|\{g_n^{-1}o\}|$ are infinite. 
Indeed, for any two distinct vertices $u$ and $v$, the map $n\mapsto \{g_n \cdot [u,v]\}$ must be finite-to-one, since $n\mapsto g_n e$ is finite-to-one for any edge $e$ between $u$ and $v$.
Thus there exists a subsequence $\{g_{n_k}\}_k$ and distinct vertices $u$ and $v$ such that $k\mapsto g_{n_k} u$ and $k\mapsto g_{n_k} v$ are both one-to-one. 
Since we also have $g^{-1}_{n_k}\to\infty/\mathcal S$, the same argument shows that either $|\{g_{n_k}^{-1} u\}_k|=\infty$ or $|\{g_{n_k}^{-1} v\}_k|=\infty$.

Now pick a subsquence $\{g_{n_k}\}$ such that $\lim g_{n_k} o = a$ and $\lim g_{n_k}^{-1} o=b$.

\textbf{Case 1:}  $a\in \Delta T\setminus T$.

Since $d(o, g_{n_k} o)=d(g_{n_k}^{-1} o, o)$, $b$ is also in $\Delta T\setminus T$. The argument in this case proceeds similarly to the argument in \cite{Tu94}. 
Let $U_a=U(a,\{e\})$ and $U_b=U(b,\{f\})$ be basic open sets in $\Delta T$, where $e$ and $f$ are edges with endpoints $e_1,e_2$  and $f_1,f_2$ respectively. For $v\in \Delta T$ denote by $\beta_v$ the geodesic from $o$ to $v$. Denote by $\langle x,y \rangle_{o}$ the Gromov product, given by the distance between $o$ and the center of the unique geoedesic tripod formed by $x,y$ and $o$. 
For all $c\notin U_b$, i.e, $[b,c]\cap [f_1,f_2]\neq \emptyset$, we have $$\liminf_{m,n\to \infty}\langle \beta_b(m), \beta_c(n) \rangle_o\leq \max\{d(o,f_1),d(o,f_2)\}=:d_1$$
Define $V(b;2d_1)=\{x\in \Delta T \ | \  \liminf_{m,n\to \infty} \langle \beta_{x}(m), \beta_{b}(n)\rangle_o\geq 2d_1 \}$. 
From the above, we have $U_b^c\subset V(b,2d_1)^c$. Now let $c\in V(b,2d_1)$. We have: $$\liminf_{m,n\to \infty} \langle \beta_a(m),g_{n_k}\beta_{c}(n) \rangle_o\geq \min\{ \liminf_{m\to \infty}\langle \beta_a(m),g_{n_k}o \rangle_o, \liminf_{n\to \infty}\langle g_{n_k}o, g_{n_k}\beta_c(n) \rangle_o \}.$$ 
Note that $\lim_{m\to\infty} \langle \beta_a(m),g_{n_k}0 \rangle_o\to\infty$ as $k\to \infty$; also, $$\langle g_{n_k}o,g_{n_k}\beta_c(n) \rangle_o=\langle o,\beta_c(n) \rangle_{g_{n_k}^{-1}o}\to\infty $$ as $k\to\infty$ since $c\notin V(b,2d_1)$. Thus $\forall R>0$, there exists $K_0(R)$ such that for all $k\geq K_0(R)$, we have $g_{n_k}(V(b,2d_1))\subset V(a,R)$. 

Now, suppose for all $k>0$, there exists $k_0>k$ and $c\in U_b^c$ such that $g_{n_k}(c)\notin U(a,\{e\})$, that is, $$\liminf_{m,n\to \infty} \langle \beta_a(n), g_{n_k}\beta_c(m) \rangle_o\leq \min\{d(o,e_1),d(o,e_2)\}=:d_2.$$ 
Take $R=2d_2$. Then $g_{n_k}(c)\notin U(a,\{f\})$ but $g_{n_k}(c)\in V(a,2d_2)$ for all $k>K(R)$, which is a contradiction. Hence, we have the conclusion in the case that $a\in \Delta T\setminus T$.

\textbf{Case 2:} $a\in T$. Let $x\in \Delta T$. We claim that $\lim_{k\to \infty} g_{n_k}(x)=a$. 

First suppose $x\in T$. For any open set $U(a,\{e\})$, we have the following: $$ |\{ n\in \mb{N}: g_n[o,x]\ni e\}|<\infty. $$ Indeed, if the above set was infinite, then there would be an infinite set $I\subset \mb{N}$ with $g_{n_{i}}^{-1}e\in [o,x]$ for $i\in I$. Since $[o,x]$ is a fixed finite length path, this contradicts the fact that $g_{n_i}$ escapes all edge stabilizers. Now, since $g_n(o)\to a$, we have that $\exists N>0$ such that for all $n\geq N$, $g_no\in U(a,\{e\})$. Hence for $n$ sufficiently large, we have that $\{e\}$ is not in the segment $[g_no,g_nx]$ and $\{e\}$ is not in the segment $[g_no,a]$, hence $\{e\}$ is not in the segment $[g_nx,a]$ which gives us what is required. 

Suppose $x\in \del T$. We have that $\{g_no\}$ is bounded, and hence let $d_0=\sup_{n} d(a,g_no)$. Given an open set $U(a,\{e\})$, let $d_1=d(a,e_1)+1$, where $d$ is the combinatorial metric on $T$. Let $N=d_0+d_1+1$. For any $n\geq N$, we apply the previous argument to the point $\beta_x (n)$ to obtain that $e\notin g_m[\alpha(0), \beta_x (n)]$ for large $m$. Moreover, due to the bound on $d(a,e)$, we have that $e\notin[\beta_x(n),x]$.

To show relative proper proximality, consider a sequence $\{g_n\}$ with $g_n\to \infty /\mathcal S$. As mentioned in \cite[Section 2]{Bow99}, if $G\curvearrowright \Delta T$ is non-elementary then the limit set is perfect, i.e., there exists a compact Hausdorff perfect set inside $\Delta T$, which implies the existence of some diffuse measure. Fix a diffuse probability measure $\eta$ on $\Delta T$, and let $a, b\in \Delta T$ and $\{g_{n_k}\}$ be as in the above arguments. By \cite[Lemma 8.3]{FLM18}, for any $h\in G$, $\lim_{k\to\infty} g_{n_k} h\eta=\delta_{a}$ and thus $X_\mathcal S=\partial_\eta G$. Moreover, there is no $G$-invariant probability measure on $\Delta T$.  Indeed, if there were a $G$-invariant probability measure $\mu$, then we see from the north-south dynamics that $\mu(\{a,b\})=1$ where $a$ and $b$ are some north and south poles. This contradicts the fact that the action is non-elementary.
\end{proof}


Applying the above proposition to the Bass-Serre tree associated with an amalgamated free product, we obtain:

\begin{cor}\label{properly proximal relative to amalgams}
Given $G_1, G_2$ and $H$ with $H\leq G_i$, $i=1,2$. Then $G=G_1\ast_H G_2$ is properly proximal relative to $H$ if $[G_i:H]\geq 3$ and $[G_j:H]\geq 2$, for $\{i,j\}=\{1,2\}$.
\end{cor}
\begin{proof}
 This follows by considering the standard action on the compactification of the Bass-Serre tree associated to the amalgamated free product. This action is non-elementary on $\Delta T$ because the boundary is infinite and orbits of the action on the boundary are infinite. Moreover, the edge stabilizers are precisely conjugates of $H$.  
\end{proof}

  The following is a direct proof of the above corollary without involving Bass-Serre theory. We would like to point out that the technical difference between non-inner amenability and proper proximality is clearly seen in this proof. 

 \begin{proof}[Alternative proof of Corollary \ref{properly proximal relative to amalgams}]
  Fix a choice of transversals $T=\{e\}\cup\{t_i\}_{i\in I}$ for cosets $\{Hx: x\in G_1\}$ and $S= \{e\}\cup\{s_j\}_{j\in J}$ for cosets $\{Hx: x\in G_2\}$.
  Denote by $p_1$ the characteristic function on the set of elements $g$ whose normal form begins with $ht_i$, where $h\in H$, $i\in I$; similarly, let $p_2$ be characteristic function on the set of elements $g$ whose normal form begins with $hs_j$, where $h\in H$, $j\in J$.

  \textbf{Claim:} $p_i\in (\ell^\infty(G)/c_0(G, H))^{G_r}$ for $i=1,2$.

  Suppose there exists some $g\in G$ and a sequence $g_n\to\infty/H$ such that $\lim_{n\to\infty} p_1(g_n)-p_1(g_ng)\neq 0$, then we may extract a subsequence, still denoted by $(g_n)$, with $\lim_{n\to\infty} p_1(g_n)-p_1(g_ng)= 1$.
  Let $g_n=h_n t_{1,n}s_{1,n}\dots t_{k_n,n}s_{k_n,n}$ be the normal form of each $g_n$, where $h_n\in H$, $t_{i,n}\in T\setminus\{e\}$ and $S_{i,n}\in S\setminus\{e\}$, with $t_{1,n}$ and $s_{k_n,n}$ possibly being $e$.
  Note that since $g_n\to\infty/H$, $\{h_n^{-1} g_n\mid n\in\mathbb N\}$ is a infinite set.
  Set $g=ht_1 s_1 \dots t_m s_m$ to be its normal form.
  As $\lim_{n\to\infty} p_1(g_n)-p_1(g_ng)= 1$, for large enough $n$ we have $t_{1,n}\in T\setminus\{e\}$ while $g_ng\notin \supp(p_1)$.
  Note that 
  \begin{align*}
      g_ng&=h_n t_{1,n}s_{1,n}\dots t_{k_n,n}s_{k_n,n}ht_1 s_1 \dots t_m s_m\\
      &=h_n h't_{1,n}'s_{1,n}'\dots t_{k_n,n}'s_{k_n,n}'t_1 s_1 \dots t_m s_m,
  \end{align*}
 
  where $t_{i,n}'\in T\setminus\{e\}$, $s_{i,n}'\in S$ are obtained as we move $h$ to the front.
  By our assumption, $g_ng\notin \supp(p_1)$, i.e., $t_{1,n}'$ would disappear after possible cancelations for all large enough $n$; however $g$ is given while $\{h_n^{-1} g_n\mid n\in\mathbb N\}$ is a infinite set, we arrive at a contradiction. This proves the claim.

  Now take $g_1, g_2\in G_i\setminus H$ and $g_3\in G_j\setminus H$, which is allowed by the assumption that $[G_i:H]\geq 3$ and $[G_j:H]\geq 2$. Notice that $g_1p_1+g_2p_1\leq p_2$ while $g_3p_2\leq p_1$ and $p_1+p_2=1$.
  If there exists a left $G$-invariant state $\varphi$ on $(\ell^\infty(G)/c_0(G, H))^{G_r}$, then $2\varphi(p_1)\leq \varphi (p_2)\leq \varphi (p_1)$, i.e., $\varphi(p_1)=\varphi(p_2)=0$ while $1=\varphi(p_1+p_2)$, which is a contradiction.
 \end{proof}

\begin{proof}[Proof of Theorem \ref{main theorem action on trees}]

Suppose $G$ is not properly proximal. Then there exists a left $G$-invariant state $\varphi$ on the space $\left((\ell^{\infty}(G)/c_0(G))^{**}\right)^{G_r}$.
Let $\{H_i\}_{i\in I}$ be a family of subgroups of $G$ such that $\bigsqcup_{i\in I}\{g H_i g^{-1}\mid g\in G\}=\{{\rm stab}(e)\mid e\in E(t)\}$.
Set $p_i=\bigvee_{t,s\in G }1_{t H_i s}$ to be a projection in $\left((\ell^{\infty}(G)/c_0(G))^{**}\right)^{G_r}$. 
From Proposition \ref{action on trees}, we obtain that $\varphi(p)=1$, where $p=\bigvee_{i\in I} p_i$. Indeed, we have that $G$ is properly proximal relative to $\mathcal S=\{{\rm stab}(e)\mid e\in E(t)\}$ and so there is no left invariant state on $\left(\left(\ell^{\infty}(G)/c_0(G,\mathcal S)\right)^{**}\right)^{G_{r}}$. But observe that   $\left(\left(\ell^{\infty}(G)/c_0(G,\mathcal S)\right)^{**}\right)^{G_{r}}= (1-p) \left(\ell^{\infty}(G)/c_0(G)\right)^{G_r}$, and hence $\varphi(1-p)=0$.
For $t\in G$ and $i\in I$, put $p_{t,i}= \bigvee_{s\in G} 1_{tH_is} \in \left((\ell^{\infty}(G)/c_0(G))^{**}\right)^{G_r}$ and we have $\bigvee_{t\in G} p_{t,i}=p_i$. 
Let $\{F_n\}$ be a sequence of increasing finite subsets such that $\cup_n F_n=G$.
For each $t\in G$, consider projections $p_{t,i,n}= 1_{tH_i F_n}\in (\ell^{\infty}(G)/c_0(G))^{**}$ and then $\lim_{n\to \infty} p_{t,i,n}=p_{t,i}$ in SOT. 

Note that $|tH_is\cap H_j|<\infty$ for any $i, j\in I$ (not necessarily distinct), $s\in G$ and $t\in G$ such that $tH_it^{-1}\neq H_j$; indeed, suppose there exists some $ths\in tH_is\cap H_j$. Then for any $th's\in  tH_is\cap H_j$, we have $ths(th's)^{-1}=t hh'^{-1}t^{-1}\in tH_i t^{-1}\cap H_j$, which is finite as for any pair of distinct edges $e,f\in E(T)$, $\mathrm{Stab}(e)\cap\mathrm{Stab}(f)$ is finite. 
Therefore we have $p_{t,i,n}p_{s,j,n}=0$ for all $t, s\in G$ such that $tH_i t^{-1}\neq sH_j s^{-1}$. Since the $(p_{t,i,n})$ and $(p_{s,j,n})$ are bounded,  $\lim_{n\to\infty}p_{t,i,n}p_{s,j,n}=p_{t,i}p_{s,j}$, i.e., $p_{t,i}p_{s,j}=0$.

Consider a $G$-equivariant unital embedding $\psi:\ell^\infty(V(T))\ni \delta_v\mapsto \sum_{e\in E_v}\delta_e/2\in\ell^\infty(E(T))$, where $E_v=\{e\in E(T)\mid v\in e\}$.
Given $e\in E(T)$, there exists a unique $i\in I$ and some $t\in G$ such that ${\rm stab}(e)=t H_i t^{-1}$. Thus $p_{t,i}$ is a unique projection corresponding to $e$ and we denote it by $p(e)$.
Therefore we obtain an embedding $\iota: \ell^\infty(E(T))\to \left((\ell^{\infty}(G)/c_0(G))^{**}\right)^{G_r}$ given by $\iota(\delta_{e})=p(e)$ and it is easy to check this embedding is unital and $G$-equivariant. 
Finally, as $C(\Delta T)\subset \ell^\infty(V(T))$, we have a $G$-invariant state $\varphi\circ\iota\circ\psi$ on $C(\Delta T)$, which contradicts the assumption that $G\curvearrowright \Delta T$ is non-elementary.
\end{proof}

In the above proof, we actually showed the following useful tool:

\begin{lem}\label{upgrading tool}

 Let $H$ be a subgroup of $G$ such that $G$ is properly proximal relative to $H$. Suppose $H$ is almost malnormal and not co-amenable in $G$, then $G$ is properly proximal. 
 
\end{lem}

For the convenience of the reader, we extract the argument here: if $G$ is not properly proximal, there exists a left $G$-invariant state $\varphi$ on the space $\left((\ell^{\infty}(G)/c_0(G))^{**}\right)^{G_r}$. Setting $p=\bigvee_{t,s\in G }1_{t H s}$ to be a projection in $\left((\ell^{\infty}(G)/c_0(G))^{**}\right)^{G_r}$, we see that   $\left(\left(\ell^{\infty}(G)/c_0(G,H)\right)^{**}\right)^{G_{r}}= (1-p) \left(\ell^{\infty}(G)/c_0(G)\right)^{G_r}$, and hence $\varphi(1-p)=0$ from the relative proper proximality hypothesis. For $t\in G$, setting $p_{t}= \bigvee_{s\in G} 1_{tHs} \in \left((\ell^{\infty}(G)/c_0(G))^{**}\right)^{G_r}$, we have $\bigvee_{t\in G} p_{t}=p$. As in the above proof, from the almost malnormality of $H$, we see that $p_g$ is orthogonal to $p_e$ for any $g\notin H$.  Therefore we obtain an embedding $\iota: \ell^\infty(G/H)\to p\left((\ell^{\infty}(G)/c_0(G))^{**}\right)^{G_r}$ given by $\iota(\delta_{g})=p_g$ and it is easy to check this embedding is unital and $G$-equivariant. Composing with $\varphi$, we obtain the co-amenability of $H<G$, which contradicts the assumption. 

\begin{lem}\label{Ozawa trick}
If $G$ is a non-trivial group and $H$ is non-amenable, then $K=G\wr H$ is properly proximal relative to $H$. 
\end{lem}
\begin{proof}
 We use a trick due to Ozawa which he uses to study bi-exactness of wreath products (see Section 15.3 in \cite{BO08}). Fix a proper length function $|\cdot|_H$ on $H$, and a proper length function $|\cdot|_G$. For $yt\in G\wr H$, where $y\in \oplus_{H}G$ and $t\in H$, define $\zeta: K\to \ell^1(H)$, given by $$ \zeta(yt)(p)= \begin{cases} 
      \min\{|p|_H,|t^{-1}p|_{H}\}+ |y(p)|_{G} & \text{if $p\in $ supp($y$)} \\
      0 & \text{if $p\notin $ supp($y$).}
   \end{cases} $$ 
From Lemma 15.3.7 and 15.3.8 in \cite{BO08}, we see that $\zeta$ satisfies  $$ \lim_{x\to \infty/H}\frac{\norm{\zeta(sxt)-s\zeta(x)}}{\norm{\zeta(x)}}=0  $$ for all $s,t\in K$.
Hence by the first computation in the proof of Lemma 15.2.6 in \cite{BO08}, we see that there exists a $K$-equivariant u.c.p. map from $\ell^\infty(K/\oplus_{H}G)$ to the relative proximal space $\left(\ell^{\infty}(K)/c_0(K,H)\right)^{{K}_{r}}$. Hence, if $K$ is not properly proximal relative to $H$, by composing with this u.c.p. map, we obtain an $H$ invariant state on $\ell^\infty(K/\oplus_{H}G)\cong \ell^\infty(H)$, which contradicts the non-amenability of $H$.
\end{proof}

\begin{proof}[Proof of Theorem \ref{wreath product}]

Observe that since $G$ is non-trivial,  $H$ is almost malnormal inside $G\wr H$. Moreover, if $H$ is non-amenable, from Lemma \ref{Ozawa trick}, and Lemma \ref{upgrading tool}, we have that $G\wr H$ is properly proximal, since $H$ is never co-amenable inside $G\wr H$. Indeed, we may first identify $(G\wr H)/H$ with $\oplus_G H$ as $H$-sets. Then consider the following map $\rho: \ell^\infty (H)\to \ell^\infty(\oplus_G H)$ given by 
$$\rho(f)(\xi)= \frac{\sum_{t\in \text{supp}(\xi)}f(t)}{|\text{supp}(\xi)|}, $$ for any $f\in \ell^\infty(H)$ and $\xi\in\oplus _GH$ and clearly $\rho$ is unital and $H$ equivariant.  

Conversely, if $H$ is amenable and infinite, then from (1) of \cite{DTW19} Theorem 3.9, we see that $G\wr H$ is inner amenable and hence not properly proximal. 
\end{proof}

\subsection{Proof of Theorem \ref{main theorem introduction 3}}
The main idea in proving Theorem \ref{main theorem introduction 3} is to consider various amalgamated product decompositions in a graph product, and obtain relative proper proximality relative to each of these amalgams. Then, we show that the Stone-Cech boundary of the graph product is filled by each of these boundary pieces coming from the amalgams, thereby showing proper proximality. The other key step  is to show that if a product of groups is properly proximal, then each of the groups has to be properly proximal. This is obtained by establishing a natural isomorphism of the proximal spaces, at the level of the double dual. The classification result is then obtained by a careful analysis of some cases, given by the radius of the graphs. 

Given two subgraphs $\Gamma_1, \Gamma_2$ of a graph $\Gamma$, denote by $\Gamma_1\cap \Gamma_2$ the subgraph of $\Gamma$ generated by the vertex set $V(\Gamma_1)\cap V(\Gamma_2)$.
As a convention, $\Gamma(G)$ is set to be the trivial group if $V(\Gamma)=\emptyset$.

\begin{lem}\label{intersection}
Let $\Gamma(G)$ be a graph product group with graph $\Gamma$ and generating groups $\{G_v\mid v\in V(\Gamma)\}$.
For any subgraphs $\Gamma_1$, $\Gamma_2$ of $\Gamma$ and $g, h\in \Gamma(G)$, we have $\Gamma_1(G)\cap g \Gamma_2(G) h\subset \cup_{i=1}^n c_i (\Gamma_1\cap \Gamma_2)(G) d_i$, for some finite subset  $F= \{c_i, d_i\mid 1\leq i\leq n\}\subset \Gamma(G)$.
\end{lem}

\begin{proof}

Given $g, h\in \Gamma(G)$ with $g= g_1\cdots g_k$ and $h= h_1\cdots h_m$ as their reduced words respectively, let $F=\{c_i,d_i\}$ be the finite set consisting of words of the form $c_i= \prod_{j}g_{i_j}$ and $d_i= \prod_{j}h_{i_k}$, where $(i_j)$ and $(i_k)$ are increasing.
We claim that $F$ is the required set. It suffices to show every $i$, we have $\Gamma_1(G)\cap c_i\Gamma_2(G)d_i\in \cup_{j=1}^n c_j (\Gamma_1\cap \Gamma_2)(G) d_j$. 
We show this by induction on  $n(t)=|\{t_j\mid 1\leq j\leq l,\ t_j\in G_v \mathrm{\ for\ some\ } v \notin \Gamma_1\cap \Gamma_2 \}|$, where $t=t_1\cdots t_l$ is a reduced word in $\Gamma_2(G)$ such that $gth\in \Gamma_1(G)$.
The claim is true for $n=0$, as $c_itd_i\in c_i(\Gamma_1\cap\Gamma_2)(G)d_i$. 

Suppose the claim is true for $n\leq K$ and let $n=K+1$. As $n=|\bigcup_{v\notin V(\Gamma_1\cap \Gamma_2)}\{t_j\mid 1\leq j\leq l, t_j\in G_{v}\}\big|$, we may assume $|\{t_j\mid 1\leq j\leq l, t_j\in G_{v_0}\}|\geq 1$ for some $v_0\notin V(\Gamma_1\cap \Gamma_2)$. 
For a word $w$ define $\hat{v}_0(w)$ to be the ordered set of letters in $w$ (ordered by the left to right order in the word $w$) that do not belong to $G_{v_0}$. For an ordered set of letters $u$ denote by $u_{\Gamma(G)}$ the product of letters of $u$ in $\Gamma(G)$. Denote by $w^*$ the reduced word of  $c_itd_i$ and note $\hat{v}_0(w^*)_{\Gamma(G)}=w^*$ as $c_itd_i\in \Gamma_1(G)$.



 \textbf{Claim:} $\hat{v}_0(c_itd_i)_{\Gamma(G)}= {c_itd_i}_{\Gamma(G)} = w^* $. 
 
 Denote the word $c_itd_i$ by $p_1p_2\cdots p_n$ and suppose $p_{i_1},\dots, p_{i_d}$ are letters that belong to $G_{v_0}$. 
 Define words $w_1=p_1\cdots p_{i_1-1} $, $w_j= p_{i_{j-1}+1}\hdots p_{i_{j}-1}$ for $1<j\leq d$ and $w_{d+1}= p_{i_d+1}\hdots p_n$. 
 Upon individually reducing these words, we obtain the words $w_i^*$. 
 Now, we have ${c_itd_i}_{\Gamma(G)}= {w^*_1p_{i_1}w^*_2p_{i_2}\hdots p_{i_d}w^*_{d+1}}_{\Gamma(G)}$. 
We may assume that ${w^*_j}_{\Gamma(G)}$ does not commute with $G_{v_0}$ for some $2\leq j \leq d$. 
Then, ${w^*_j}_{\Gamma(G)}$ neither commute with $a= {w^*_1p_{i_1}\hdots p_{i_{j-1}}}_{\Gamma(G)}$ nor with $b={p_{i_j}w^*_{j+1}\hdots w^*_{d+1}}_{\Gamma(G)}$. 
It follows that $p_jp_{j+1}\hdots p_d=1$ and $p_1\hdots p_{j-1}=1$. 
Suppose otherwise and then ${c_itd_i}_{\Gamma(G)}={aw^*_jb}_{\Gamma(G)}$ while $a,b\notin \ker(\pi_{v_0})$, where $\pi_{v_0}$ is the canonical surjection from $\Gamma(G)\to G_{v_0}$ . Since $w^*_j$ doesn't commute with $G_{v_0}$ we see ${c_itd_i}_{\Gamma(G)}\notin \Gamma_1(G)$ which is a contradiction. We proceed recursively to conclude the claim.  
 
 Now it suffices to show $\hat{v}_0(c_itd_i) \in  \cup_{j=1}^n c_j (\Gamma_1\cap \Gamma_2)(G) d_j $, and this follows from the inductive hypothesis as $\hat{v_0}(c_itd_i)= c_{i'} \hat{v_0}(t) d_{i'}$ for some other $i'$, and $n(\hat{v_0}(t))<n(t)=K+1 $.
\end{proof}

\begin{lem}\label{boundary pieces associated to stars union to the whole piece}
Let $\Gamma$ be a graph. Then, $$\bigcup_{v\in V(\Gamma)}X_{\Gamma_{S_v}(G)}= \Delta \Gamma(G).$$

\end{lem}
\begin{proof}
Set $V(\Gamma)=\{v_i\mid i=1, \dots ,n\}$ and $\Gamma_i=\Gamma_{S_{V_i}}$ for each $i$. Note that this statement is equivalent to $$\bigcap_{i=1}^n c_0(\Gamma(G),\Gamma_i(G))=c_0(\Gamma(G))$$ and thus it suffices to show $|\bigcap_{i=1}^n s_i \Gamma_i (G) t_i|<\infty$ for any $s_i, t_i\in \Gamma(G)$.

By Lemma \ref{intersection}, $\bigcap_{i=1}^n s_i \Gamma_i (G) t_i\subset \bigcup_{j=1}^m g_j (\bigcap_{i=1}^n \Gamma_i)(G) h_j$ for some finite subset $\{g_j, h_j\mid 1\leq j\leq m\}$ of $\Gamma(G)$.
Note that $\cap_{i=1}^n S_{V_i}=\emptyset$ and thus $(\bigcap_{i=1}^n \Gamma_i)(G)=\{e\}$, i.e., $\bigcap_{i=1}^n s_i \Gamma_i (G) t_i$ is finite.
\end{proof}






\begin{prop}\label{main theorem introduction 1}
Let $\Gamma$ be a finite  graph. The following are equivalent:
\begin{enumerate}
    \item $\Gamma$ satisfies the following condition: $r(\Gamma)$ is at least 2 
    and there does not exist a pair of vertices $v_1,v_2\in V(\Gamma)$ such that $(u,v_i)\in E(\Gamma)$ for all $u\in V(\Gamma)\setminus \{v_1,v_2\}$, and $i=1,2$. 
    \item For all choices of non-trivial groups $G_v$, $v\in V(\Gamma)$, the graph product $\Gamma(G)$ is properly proximal.
\end{enumerate}
\end{prop}

\begin{proof}[Proof of Proposition \ref{main theorem introduction 1}]
First, we show $(2)\implies (1)$. Suppose there is a radius 2 graph such that there is a pair of vertices $v_1,v_2$ satisfying  $(u,v_i)\in E(\Gamma)$ for all $u\in V(\Gamma)\setminus \{v_1,v_2\}$, and $i=1,2$, then one can simply consider $G_{v_1}= \mathbb Z/2\mathbb {Z}= G_{v_2}$. Since $(v_1,v_2)\notin E(\Gamma)$, we see that  $\Gamma(G)= \mb{Z}_2*\mb{Z}_2\times \Gamma_{V(\Gamma)\setminus \{v_1,v_2\}}(G)$ is inner amenable, and therefore not properly proximal by Proposition 4.11 in \cite{BIP18}.

For $(1)\implies (2)$, observe that by Lemma \ref{graph products are amalgamated free products}, we have a natural amalgamated free product decomposition: $\Gamma(G)\cong G_1 *_{H} G_2 $, using the same notation.
Since $\Gamma$ satisfies (1),  we have $\text{max}\{[G_1:H], [G_2:H]\}=\infty$ and $H$ is neither $G_1$ nor $G_2$.
Then it follows immediately from Lemma \ref{boundary pieces associated to stars union to the whole piece} and Corollary \ref{properly proximal relative to amalgams}.
\end{proof}

It follows directly that for an arbitrary choice of non-trivial groups indexed by vertices, $\Gamma(G)$ is properly proximal, provided that $\Gamma$ is an irreducible graph or a graph satisfying $r(\Gamma)\geq 3$.

\begin{prop}\label{main theorem introduction 2}
Let $\Gamma$ be such that $r(\Gamma)\geq 2$. Then $\Gamma(G)$ is either properly proximal or contains an infinite amenable summand.
\end{prop}
\begin{proof}
The only case to check is if $\Gamma$ has radius 2 and it doesn't satisfy condition (1) of Proposition \ref{main theorem introduction 1}. That is, there exists vertices $v_1,v_2$ such that $S= V(\Gamma)\setminus \{v_1,v_2\}$ is connected to $v_1$ and $v_2$. Then since $v_1$ is not connected to $v_2$ (else it violates the radius 2 condition), we have $\Gamma(G)\cong \left(G_{v_1}*G_{v_2}\right) \times \Gamma_{S}(G)$. The left summand is either infinite amenable ($\mb{Z}_2*\mb{Z}_2$) or properly proximal (since they are non-trivial free products). The above decomposition can be iterated on $S$ and further on, using the hypothesis. The result follows from the fact that properly proximal groups cannot have infinite amenable summands (\cite[Proposition 4.11]{BIP18}).
\end{proof}

We thank J. Peterson for suggesting the proof of  the following proposition. A generalization of the result below will appear in \cite{DKEP21}. The converse already appears as Proposition 4.10 as in \cite{BIP18}.
\begin{prop}\label{Products are properly proximal iff summands are properly proximal}
If $G=G_1\times G_2$ is properly proximal, then both $G_1$ and $G_2$ are properly proximal.
\end{prop}
\begin{proof}
Define the projection $p=\bigvee_{g\in G_2}1_{G_1\times\{g\}}\in \left(\ell^{\infty}(G)/c_0(G)\right)^{**}$. We claim the following isomorphism holds: $$\bigoplus_{G_2}^\infty (\ell^\infty G_1/c_0(G_1))^{**}\cong p(\ell^\infty G/c_0(G))^{**},$$ 
where the direct sum is equipped with $\ell^\infty$-norm.
Indeed, for any finite subset $F\subset G_2$, consider the isomorphism
\[
\Theta_F:\bigoplus_{F} \ell^\infty G_1\ni (f_h)_{h\in F}\mapsto \tilde f_F\in 1_{G_1\times F}\ell^\infty G,
\]
where $\tilde f_F(g,h)=f_h(g)$ for $g\in G_1,$ $h\in F$.
Composing with the quotient map, one has an isomorphism $\Theta_F:\bigoplus_{F} (\ell^\infty G_1/c_0(G_1))\to 1_{G_1\times F}(\ell^\infty G/c_0(G))$.
Furthermore, we may extend this map to a weak$^\ast$ continuous isomorphism $\tilde\Theta_F: \bigoplus_{F} (\ell^\infty G_1/c_0(G_1))^{**}\to 1_{G_1\times F}(\ell^\infty G/c_0(G))^{**}$.

For each $(f_h)_{h\in G_2}\in \bigoplus_{G_2} (\ell^\infty G_1/c_0(G_1))^{**}$,
define $\tilde\Theta ((f_h))$ to be the SOT limit point of the net $\{\tilde\Theta_F((f_h)_{h\in F})\}_F$ in $p(\ell^\infty G/c_0(G))^{**}$, where $F$ ranges over all finite subsets of $G_2$;
conversely, for $p g\in p(\ell^\infty G/c_0(G))^{**}$, set $\tilde\Theta^{-1}(pg)$ to be the weak$^\ast$ limit of $\{\tilde\Theta_F^{-1}(1_{G_1\times F} g)\}_F$ in $\bigoplus_{G_2} (\ell^\infty G_1/c_0(G_1))^{**}$.
Thus $\tilde \Theta$ implements the required isomorphism, which is also $G$-equivariant by construction.

Now by taking the $G$-right invariant subspaces in both sides, we have $$\big((\ell^\infty G_1/c_0(G_1))^{**}\big)^{{G_1}_r}=\left(\bigoplus_{G_2}^\infty (\ell^\infty G_1/c_0(G))^{**}\right)^{G_r}\cong p\left((\ell^\infty G/c_0(G))^{**}\right)^{G_r}.$$ 
Now suppose there is a left $G_1$-invariant state $\varphi$ on $\big((\ell^\infty G_1/c_0(G_1))^{**}\big)^{{G_1}_r}$, then $\varphi\circ\tilde\Theta^{-1}$ is a left $G$-invariant state on $p\left((\ell^\infty G/c_0(G))^{**}\right)^{G_r}$. Set $\psi(f)=\varphi\circ\tilde\Theta^{-1}(pf)$ for $f\in \left((\ell^\infty G/c_0(G))^{**}\right)^{G_r}$ and we obtain a $G$-invariant state, which contradicts the proper proximality of $G$.
\end{proof}
\begin{proof}[Proof of Theorem \ref{main theorem introduction 3}]
 Suppose $\Gamma$ satisfies condition (1) in Proposition \ref{main theorem introduction 1}, then clearly $\Gamma(G)\in \mc{G}_{pp} $ if and only if $\Gamma(G)$ is properly proximal. If $\Gamma$ is of radius 2, and it doesn't satisfy condition (1) of Theorem \ref{main theorem introduction 1}, then $\Gamma(G)$ can be decomposed as $\Gamma_{\{v_1,v_2\}}(G)\times \Gamma_{\{v_1,v_2\}^c}(G)$. Hence by Proposition \ref{Products are properly proximal iff summands are properly proximal}, it follows that $\Gamma(G)$ is properly proximal if and only if $|G_{v_i}|\geq 3  $ for some $i\in \{1,2\}$, and $\Gamma_{\{v_1,v_2\}^c}(G)\in \mc{G}_{pp}$. Finally suppose $\Gamma$ has radius 1. Then the group decomposes as $\Gamma(G)= \Gamma_{v}(G)\times \Gamma_{\{v\}^c}(G)$, whence we apply the same argument as before. 
\end{proof}

\begin{remark}
In \cite{HHL20} it is shown that proper CAT(0) cubical complex groups are properly proximal. Graph products admit natural actions on CAT(0) cube complexes due to \cite{Mei96} and \cite{Dav98}, however these actions are in general, not proper. Thus Theorem \ref{main theorem introduction 3} cannot be deduced directly from \cite{HHL20}. 
\end{remark}




\bibliographystyle{amsalpha}
\bibliography{ref}

\providecommand{\bysame}{\leavevmode\hbox to3em{\hrulefill}\thinspace}
\providecommand{\MR}{\relax\ifhmode\unskip\space\fi MR }
\providecommand{\MRhref}[2]{%
  \href{http://www.ams.org/mathscinet-getitem?mr=#1}{#2}
}
\providecommand{\href}[2]{#2}
\begin{thebibliography}{DTDW19}

\bibitem[BH99]{BH99}
Martin~R. Bridson and Andr\'{e} Haefliger, \emph{Metric spaces of non-positive
  curvature}, Fundamental Principles of Mathematical Sciences, vol. 319,
  Springer-Verlag, Berlin, 1999. \MR{1744486}

\bibitem[BIP18]{BIP18}
R\'emi Boutonnet, Adrian Ioana, and Jesse Peterson, \emph{Properly proximal
  groups and their von {N}eumann algebras}, To appear in Ann. Sci. Ec. Norm.
  Super. (2018).

\bibitem[BO08]{BO08}
Nathanial~P. Brown and Narutaka Ozawa, \emph{{$C^*$}-algebras and
  finite-dimensional approximations}, Graduate Studies in Mathematics, vol.~88,
  American Mathematical Society, Providence, RI, 2008. \MR{2391387}

\bibitem[Bow99]{Bow99}
B.~H. Bowditch, \emph{Convergence groups and configuration spaces}, Geometric
  group theory down under ({C}anberra, 1996), de Gruyter, Berlin, 1999,
  pp.~23--54. \MR{1714838}

\bibitem[Bow12]{BOW12}
\bysame, \emph{Relatively hyperbolic groups}, Internat. J. Algebra Comput.
  \textbf{22} (2012), no.~3, 1250016, 66. \MR{2922380}

\bibitem[Cas20]{Ca20}
Martijn Caspers, \emph{Absence of {C}artan subalgebras for right-angled {H}ecke
  von {N}eumann algebras}, Anal. PDE \textbf{13} (2020), no.~1, 1--28.
  \MR{4047640}

\bibitem[CdSS18]{CdSW}
Ionut Chifan, Rolando de~Santiago, and Wanchalerm Sucpikarnon, \emph{Tensor
  product decompositions of {II}$_1$ factors arising from extensions of
  amalgamated free product groups}, Comm. Math. Phys. \textbf{364} (2018),
  1163--1194.

\bibitem[CE21]{chifan2021cartan}
Ionut Chifan and Srivatsav~Kunnawalkam Elayavalli, \emph{Cartan subalgebras in
  von neumann algebras associated with graph product groups}, 2021.

\bibitem[Dah03]{Dah03}
Fran\c{c}ois Dahmani, \emph{Combination of convergence groups}, Geom. Topol.
  \textbf{7} (2003), 933--963. \MR{2026551}

\bibitem[Dav98]{Dav98}
Michael~W. Davis, \emph{Buildings are {${\rm CAT}(0)$}}, Geometry and
  cohomology in group theory ({D}urham, 1994), London Math. Soc. Lecture Note
  Ser., vol. 252, Cambridge Univ. Press, Cambridge, 1998, pp.~108--123.
  \MR{1709955}

\bibitem[DKEP21]{DKEP21}
Changying Ding, Srivatsav Kunnawalkam~Elayavalli, and Jesse Peterson,
  \emph{Properly proximal von {N}eumann algebras}, 2021, in preparation.

\bibitem[DTDW19]{DTW19}
Bruno Duchesne, Robin Tucker-Drob, and Phillip Wesolek, \emph{{CAT}(0) cube
  complexes and inner amenability}, To appear in Groups. Geom. Dyn. (2019).

\bibitem[FLM18]{FLM18}
Talia Fern\'{o}s, Jean L\'{e}cureux, and Fr\'{e}d\'{e}ric Math\'{e}us,
  \emph{Random walks and boundaries of {$\rm CAT(0)$} cubical complexes},
  Comment. Math. Helv. \textbf{93} (2018), no.~2, 291--333. \MR{3811753}

\bibitem[Gre90]{Gr90}
Elisabeth~Ruth Green, \emph{Graph products of groups}, Ph.D. thesis, University
  of Leeds, 1990.

\bibitem[HHL20]{HHL20}
Camille Horbez, Jingyin Huang, and Jean Lécureux, \emph{Proper proximality in
  non-positive curvature}, arXiv:2005.08756, 2020.

\bibitem[Ioa15]{Io15}
Adrian Ioana, \emph{Cartan subalgebras of amalgamated free product {II}$_1$
  factors. with an appendix joint with {Stefaan Vaes}.}, Ann. Sc. Ec. Norm.
  Super. \textbf{48} (2015), no.~4, 71--130.

\bibitem[Ioa18]{Io18}
\bysame, \emph{Rigidity for von neumann algebras}, Proceedings of the
  {I}nternational {C}ongress of {M}athematicians. {V}olume {II}, 2018,
  pp.~1635--1668.

\bibitem[IPR19]{IsPeRu19}
Ishan Ishan, Jesse Peterson, and Lauren Ruth, \emph{Von {N}eumann equivalence
  and properly proximal groups}, arXiv:1910.08682, 2019.

\bibitem[Mei96]{Mei96}
John Meier, \emph{When is the graph product of hyperbolic groups hyperbolic?},
  Geom. Dedicata \textbf{61} (1996), no.~1, 29--41. \MR{1389635}

\bibitem[OP10]{OzPo10}
Narutaka Ozawa and Sorin Popa, \emph{On a class of {II}{$_1$} factors with at
  most one cartan subalgebra, {II}}, Amer. J. Math. \textbf{132} (2010), no.~3,
  841--866.

\bibitem[Osi06]{Osi04}
D.~V. Osin, \emph{Relative {D}ehn functions of amalgamated products and
  {HNN}-extensions}, Topological and asymptotic aspects of group theory,
  Contemp. Math., vol. 394, Amer. Math. Soc., Providence, RI, 2006,
  pp.~209--220. \MR{2216718}

\bibitem[PV14]{PoVa14a}
Sorin Popa and Stefaan Vaes, \emph{{Unique Cartan decomposition for $II_1$
  factors arising from arbitrary actions of free groups}}, Acta Mathematica
  \textbf{212} (2014), no.~1, 141 -- 198.

\bibitem[Rec17]{Re17}
Eric Reckwerdt, \emph{Weak amenability is stable under graph products}, J.
  Lond. Math. Soc. (2) \textbf{96} (2017), no.~1, 133--155. \MR{3687943}

\bibitem[Ser80]{Ser80}
Jean-Pierre Serre, \emph{Trees}, Springer-Verlag, Berlin-New York, 1980,
  Translated from the French by John Stillwell. \MR{607504}

\bibitem[Tom21]{tomar}
Ravi Tomar, \emph{Boundaries of graphs of relatively hyperbolic groups with
  cyclic edge groups}.

\bibitem[Tuk94]{Tu94}
Pekka Tukia, \emph{Convergence groups and {G}romov's metric hyperbolic spaces},
  New Zealand J. Math. \textbf{23} (1994), no.~2, 157--187. \MR{1313451}

\bibitem[Vae13]{Vae13}
Stefaan Vaes, \emph{{Normalizers inside Amalgamated Free Product von Neumann
  Algebras}}, {Publications of the Research Institute for Mathematical
  Sciences} \textbf{50} (2013).

\end{thebibliography}

\end{document}